\documentclass{article}
\usepackage[utf8]{inputenc}
\usepackage[english]{babel} 
\usepackage{enumitem}
\usepackage{amssymb}
\usepackage{amsmath}
\usepackage{amsthm}
\usepackage{txfonts}
\usepackage{mathdots}
\usepackage[a4paper, total={6in, 8in}]{geometry}
\usepackage[classicReIm]{kpfonts}
\usepackage[dvips]{graphicx}
\usepackage[dvipsnames]{xcolor}
\newtheorem{theorem}{Theorem}[section]
\newtheorem{cor}[theorem]{Corollary}
\newtheorem{lem}[theorem]{Lemma}

 \newtheorem{defn}{Definition}[section]

 \newtheorem{notn}{Notation}[section]

\newcommand{\R}{\mathbb{R}}

\newcommand{\N}{\mathbb{N}}

\title{Semilinear degenerate elliptic equation in the presence of singular nonlinearity}
\author{Kaushik Bal and Sanjit Biswas }
\begin{document}

\maketitle

\section*{ABSTRACT}
Given $\Omega(\subseteq\R^{1+m})$, a smooth bounded domain and a nonnegative measurable function $f$ defined on $\Omega$ with suitable summability. In this paper, we will study the existence and regularity of solutions to the quasilinear degenerate elliptic equation with a singular nonlinearity given by:
\begin{align}
    -\Delta_\lambda u&=\frac{f}{u^{\nu}} \text{ in }\Omega\nonumber\\
   &u>0 \text{ in } \Omega\nonumber\\
   &u=0 \text{ on } \partial\Omega\nonumber
\end{align}
where the operator $\Delta_\lambda$ is given by $$\Delta_\lambda{u}=u_{xx}+|x|^{2\lambda}\Delta_y{u};\,(x,y)\in \R\times\R^m $$
is known as the Grushin operator.

\tableofcontents
\begin{center}
     \section{INTRODUCTION}\label{sec1}
\end{center}
In this paper, we are interested in the semilinear elliptic problem, whose model is given by 
 \begin{align}\label{maineq}
    -\Delta_\lambda u&=\frac{f}{u^\nu} \text{ in }\Omega\\
   &u>0 \text{ in } \Omega\\
   &u=0 \text{ on } \partial\Omega
\end{align}
where the operator $\Delta_\lambda$ is given by $$\Delta_\lambda{u}=u_{xx}+|x|^{2\lambda}\Delta_y{u};\;\lambda\geq 0$$
is known as the Grushin operator. $\Delta_y$ denotes the Laplacian operator w.r.t $y$ variable. $\Omega\subseteq\R^{1+m}$ is a $\Lambda-$connected bounded open set (definition provided in the next section) and  $X=(x,y)\in\Omega$, $x\in\R$, $y=(y_1,y_2,...,y_m)\in\R^m$, $m\geq1$. Here $\nu>0$ is a positive real number, and $f$ is a nonnegative measurable function lying in some Lebesgue space.\\
To understand the context of our study, we start by looking at available literature concerning (\ref{maineq}). Starting with the now classical work by Crandall et al. \cite{CRT} where the case $\lambda=0$ was considered and showed to have a unique solution in $C^2(\Omega)\cap C(\bar{\Omega})$ such that the solution behaves like some power of the distance function near the boundary, a plethora of work followed provided $f\in C^{\alpha}(\Omega)$. Of particular significance is the work of Lazer-Mckenna, where the solution was shown to exist in $H_0^1(\Omega)$ if and only if $0<\delta<3$. When $f\in L^1(\Omega)$, Boccardo and Orsina \cite{BLO} proved if $0<\nu\leq 1$ then there exist a solution of (\ref{maineq}) in $H^1_0(\Omega)$ and for $\nu>1$ there exist a solution $u\in H^1_{loc}(\Omega)$ such that $u^\frac{\nu+1}{2}\in H^1_0(\Omega)$ among other regularity results. The p-laplacian case was settled by \cite{CST}, where existence, uniqueness, and some regularity results were proved.

In this paper, we would like to relook at the equation (\ref{maineq}) by replacing the Laplacian with a degenerate elliptic equation whose prototype is given by Grushin Laplacian $\Delta_\lambda$. We will prove the existence and regularity results analog to \cite{BLO}. It is worth pointing out that there are several issues when degeneracy is introduced. If the distance between the domain $\Omega$ and the plane $x=0$ is positive, then the Grushin operator will become uniformly elliptic in $\Omega$, and in this case, the problem is settled in \cite{BLO}. We assume the domain $\Omega$ intersects the $x=0$ plane, thus degenerating the operator in $\Omega$. To handle this kind of degeneracy, assuming that $\Delta_\lambda$ admits a uniformly elliptic direction, we discuss the solvability of (\ref{maineq}) in the weighted degenerate Sobolev space $H^{1,\lambda}(\Omega)$ which is defined in \cite{FS, FL1}. We would also need to have a notion of convergence of sequence in the space $H^{1,\lambda}(\Omega)$ for which Monticelli-Payne \cite{MP} introduced the concept of a quasi-gradient, hence providing a proper representation of elements of $H^{1,\lambda}(\Omega)$. Another issue is the lack of availability of the Strong Maximum Principle, which we showed to hold using weak Harnack inequality of Franchi-Lanconelli \cite[Theorem 4.3]{FL2} valid for
$d-$metric on $\Omega$ provided $\lambda\geq 1$ and assuming that $\Omega$ is $\Lambda-$connected (definition is provided in the next section). We conclude our study with a brief discussion of how singular variable exponent for Grushin Laplacian may be handled, whose Laplacian counterpart can be found in Garain-Mukherjee \cite{PT}. For further reading into the topic, one may look at the papers \cite{BBG, BG1, BG, BGM, MD} and the references therein. \\
\begin{notn} 
Throughout the paper, if not explicitly stated, $C$ will denote a positive real number depending only on $\Omega$ and $N$, whose value may change from line to line. We denote by $\langle.,.\rangle$ the Euclidean inner product on $\R^n$ and denote by $|A|:=\sup_{|\xi|=1}\langle A\xi,\xi\rangle$ the norm of a real, symmetric $N\times N$ matrix $A$. The Lebesgue measure of $S\subset \R^N$ is denoted by $|S|$. The H\"{o}lder conjugate of $r\geq 1$ is denoted by $r'$.
 \end{notn}
This paper is organized into seven sections. Section \ref{sec2} discusses functional, analytical settings related to our problem and a few related results. We state our main results in section \ref{sec3}. Section \ref{sec4} and \ref{sec5} are devoted to proving a few auxiliary results. We prove our main results in section \ref{sec6}. Finally, in section \ref{sec7}, we consider the variable singular exponent case.

\begin{center}
\section{PRELIMINARIES AND FEW USEFUL RESULTS} \label{sec2}   
\end{center}
We define a few crucial notions, and the metric introduced in Franchi-Lanconelli \cite{FL2}.

 \begin{defn}
  An open subset $\Omega(\subset\R^N)$ is said to be $\Lambda-$connected if for every $X,Y\in\Omega$, there exists a continuous curve lying in $\Omega$ which is piecewise an integral curve of the vector fields  $\pm \partial_x,\pm|x|^{\lambda}\partial_{y_1},...,\pm|x|^{\lambda}\partial_{y_m}$ connecting $X$ and $Y$.
 \end{defn}
  Note that every $\Lambda-$connected open set in $\R^N$ is connected. We denote by $P(\Lambda)$ the set of all continuous curves which are piecewise integral curves of the vector fields $\pm \partial_x,\pm|x|^{\lambda}\partial_{y_1},...,\pm|x|^{\lambda}\partial_{y_m}$. Let $\gamma:[0,T]\to\Omega$ is an element in $P(\Lambda)$ and define $l(\gamma)=T$.
  \begin{defn}
  Let $X,Y\in\Omega$, we define a new metric $d$ on $\Omega$ by $d(X,Y)=\inf\{l(\gamma):\gamma\in P(\Lambda)$ connecting $X$ and $Y$\}.
  \end{defn}
 The $d-$ball around $X\in\Omega$ with radius $r>0$ is denoted by $S_d(X,r)$ and is given by $S_d(X,r)=\{Y\in\Omega:d(X, Y)<r$\}. (\cite[\text{Proposition 2.9}]{FL1}) ensures that the usual metric is equivalent to the $d$ in $\Omega$.

Let $N=k+m$ and $\Omega\subseteq\R^N$ be a bounded domain. Let $A=\left(\begin{array}{cc}
    I_k  &  O\\
    O  & |x|^{2\lambda}I_m
 \end{array}\right)$ and define the set $$V_A(\Omega)=\{u\in C^1(\Omega) |  \int_\Omega |u|^p\,dX + \int_\Omega \langle A\nabla u,\nabla u\rangle^\frac{p}{2}\,dX <\infty\}$$
 Consider the normed linear spaces $(V_A(\Omega),\|.\|)$ and $(C^1_0(\Omega),\|.\|_0)$ where $$\|u\|=(\int_\Omega |u|^p\,dX + \int_\Omega \langle A\nabla u,\nabla u\rangle^\frac{p}{2}\,dX)^\frac{1}{p}$$ and $$\|u\|_0=(\int_\Omega \langle A\nabla u,\nabla u\rangle^\frac{p}{2}\,dX)^\frac{1}{p}$$
Now $W^{1,\lambda,p}(\Omega)$ and $W^{1,\lambda,p}_0(\Omega)$ is defined as the completion of $(V_A(\Omega),\|.\|)$ and $(C^1_0(\Omega),\|.\|_0)$ respectively. Each element $[\{u_n\}]$, of the Banach space $W^{1,\lambda,p}(\Omega)$ is a class of Cauchy sequence in $(V_A(\Omega),\|.\|)$ and $\|[\{u_n\}]\|=\lim_{n\to\infty}\|u_n\|$. A function $u$ is said to be in $W^{1,\lambda,p}_{loc}(\Omega)$ if and only if $u\in W^{1,\lambda,p}(\Omega')$ for every $\Omega'\Subset\Omega$. For more information, one can look into Monticelli-Payne \cite{MP}.\\
  The following theorem proves that $\|.\|_0$ and $\|.\|$ are equivalent norm on $W^{1,\lambda,p}_0(\Omega)$.
 \begin{theorem}(\text{Poincar\'{e} Inequality})(Monticelli-Payne \cite [\text{ Theorem 2.1}] {MP})
  Let $\Omega\subset \R^N$ be a bounded domain, and $A$ is given as above. Then for any $1\leq p<\infty$ there exists a constant $C_p=C(N,p,\|A\|_\infty,d(\Omega))>0$ such that $$\|u\|_{L^p(\Omega)}^p
  \leq C_p\int_\Omega \langle A\nabla u,\nabla u\rangle^\frac{p}{2}\; dX\;\mbox{ for all $u\in C^1_0(\Omega)$}$$\\
  where $d(\Omega)$ denotes the diameter of $\Omega$.
 \end{theorem}
  Now the suitable representation of an element of $W^{1,\lambda,p}(\Omega)$ and $W^{1,\lambda,p}_0(\Omega)$ is given by the following theorem, whose proof follows exactly that of Monticelli-Payne where it is done for $p=2$.
  \begin{theorem}(Monticelli-Payne \cite [\text{ Theorem 2.1}] {MP}))\label{REP}
   Let $\Omega \subset \R^N$ be a bounded open set, and $A$ is given as above. Then for every $[\{u_n\}]\in W^{1,\lambda,p}(\Omega)$ there exists unique $u\in L^p(\Omega)$ and $U\in (L^p(\Omega))^N$ such that the following properties hold
   \begin{enumerate}[label=(\roman*)]
       \item $u_n\rightarrow u$ in $L^p(\Omega)$ and $\sqrt A \nabla u_n\rightarrow U$ in $(L^p(\Omega))^N$.
       \item $\sqrt{A}^{-1}U$ is the weak gradient of $u$ in each of the component of $\Omega\setminus\Sigma$
       \item If $|[\sqrt{A}]^{-1}|\in L^{p'}(\Omega)$ then $[\sqrt{A}]^{-1}U$ is the weak gradient of $u$ in $\Omega$.
       \item One has $$ \|[{u_n}]\|^p=\|u\|_{L^p(\Omega)}^p+\|U\|_{(L^p(\Omega))^N}^p$$
   \end{enumerate}
   where $\Sigma=\{X\in\Omega : \text{det}[A(X)]=0\}$, $p'=\frac{p}{p-1}$.
  \end{theorem}
  
  \begin{proof}
   Let $[\{u_n\}]\in W^{1,\lambda,p}$. So $[\{u_n\}]$ is a Cauchy sequence in $(V_A,\|.\|)$. Clearly $\{u_n\}$ and $\{\sqrt{A}\nabla u_n\}$ are Cauchy in $L^p(\Omega)$ and $L^p(\Omega)^N$. Hence there exists $u\in L^p(\Omega)$ and $U\in L^p(\Omega)^N$ such that $u_n\to u$ in $L^p(\Omega)$ and $\{\sqrt{A}\nabla u_n\}\to U$ in $L^p(\Omega)^N$ as $n\to\infty$. If $[\{u_n\}]=[\{v_n\}]$ and $\{\sqrt{A}\nabla u_n\}\to U$, $\{\sqrt{A}\nabla v_n\}\to V$ in $L^p(\Omega)^N$ as $n\to\infty$. Then 
   \begin{align*}
       \|U-V\|_{L^p(\Omega)^N}&\leq  \|\sqrt{A}\nabla u_n-U\|_{L^p(\Omega)^N}+ \|\sqrt{A}\nabla u_n-\sqrt{A}\nabla v_n\|_{L^p(\Omega)^N}+  \|\sqrt{A}\nabla v_n-V\|_{L^p(\Omega)^N}\\
       &\to 0 \text{ as $n\to\infty$}
   \end{align*}
  which implies $U=V$ a.e in $\Omega$. So $U$ does not depend on the representative of the class $[\{u_n\}]$.
  Let $\phi\in C^\infty_0(\Omega)$. Since $u_n\to u$ in $L^p(\Omega)$ so
  $u_n$ converges to $u$ in the distributional sense as well. As $u_n\in C^1(\Omega)$ so $$\int_\Omega u_n \nabla \phi dx=-\int_\Omega \phi\nabla u_n dx$$
   Taking limit $n\to\infty$ we have $$ \int_\Omega u \nabla \phi dx=-\lim_{n\to \infty}\int_\Omega \phi\nabla u_n dx= -\lim_{n\to \infty}\int_\Omega \phi\sqrt{A}^{-1}\sqrt{A}\nabla u_n dx$$
   Hence if $|\phi\sqrt{A}^{-1}|\in L^{p'}(\Omega)$ then 
   \begin{align}\label{San1}
       \int_\Omega u \nabla \phi dx=-\int_\Omega \phi\sqrt{A}^{-1} U dx
   \end{align}
   If support of $\phi$ is contained in a component of $\Omega\setminus\Sigma$ then $|\phi\sqrt{A}^{-1}|\in L^{p'}(\Omega)$. By using (\ref{San1}) we can conclude  that $\sqrt{A}^{-1}U$ is the weak gradient of $u$ in that component of $\Omega\setminus\Sigma$. Hence (ii) is proved. Also, if $|\sqrt{A}^{-1}|\in L^{p'}(\Omega)$ then (\ref{San1}) is true for every $\phi\in C^\infty_0(\Omega)$. So $\sqrt{A}^{-1}U$ is the weak gradient of $u$ in $\Omega$. Which proves (iii).\\
   For $[\{u_n\}]\in W^{1,\lambda,p}(\Omega)$,
   \begin{align*}
       \|[\{u_n\}]\|^p=\lim_{n\to\infty} (\|u_n\|_{L^p(\Omega)}^p+\|\sqrt{A}\nabla u_n\|_{L^p(\Omega)^N}^p)=(\|u\|_{L^p(\Omega)}^p+\|U\|_{L^p(\Omega)^N}^p)
   \end{align*}
   Hence (iv) is proved.
  \end{proof}
  Using the above theorem, we have the following embedding theorem.
   \begin{cor}
  The space $W^{1,\lambda,p}(\Omega)$ is continuously embedded into $L^p(\Omega)$.
  \end{cor}
  \begin{proof}
   Define the map $T:W^{1,\lambda,p}(\Omega)\to L^p(\Omega)$ by $T([\{u_n\}])=u$. $T$ is a bounded linear map. \\Claim: $T$is injective. Let $u=0$. If we can prove $U=0$, then we are done. Since $\Sigma$ has measure zero, we can prove that $U=0$ a.e in each component of $\Omega\setminus\Sigma$. Let $\Omega'$ be a component of $\Omega\setminus\Sigma$. By the above theorem for every $\phi\in C^\infty_0(\Omega')$
    \begin{align*}\label{San1}
       \int_{\Omega'} \phi\sqrt{A}^{-1} U dx=-\int_{\Omega'} u \nabla \phi dx=0
   \end{align*} 
   which ensures us $\sqrt{A}^{-1}U=0$ a.e in $\Omega'$. So $U=0$ a.e in $\Omega'$.
  \end{proof}
  Henceforth we use the notation $u$ for the element $[\{u_n\}]\in W^{1,\lambda,p}(\Omega)$ or$[\{u_n\}]\in W^{1,\lambda,p}_0(\Omega)$ which is determined in Theorem (\ref{REP}). Using the properties of $U\in(L^p(\Omega))^N$ in the theorem we introduce the following definition:
 \begin{defn}
   For $u\in W^{1,\lambda,p}(\Omega)$ we denote the weak quasi gradient of $u$ by $\nabla^* u$ and defined by $$\nabla^* u:=(\sqrt A)^{-1} U$$ which is a vector-valued function defined almost everywhere in $\Omega$.
 \end{defn}
 Also for $u\in W^{1,\lambda,p}(\Omega)$,
 \begin{align*}
     \|u\|^p&=\|u\|_{L^p(\Omega)}^p+\|\sqrt{A}\nabla^*u\|_{L^p(\Omega)}^p\\
     &=\int_\Omega|u|^p dx + \int_\Omega \langle A\nabla^*u,\nabla^*u\rangle^\frac{p}{2}.
 \end{align*}
 We define $H^{1,\lambda}(\Omega):=W^{1,\lambda,2}(\Omega)$ and $H^{1,\lambda}_0(\Omega):=W^{1,\lambda,2}_0(\Omega)$. $(H^{1,\lambda}(\Omega),\|.\|)$ and $(H^{1,\lambda}_0(\Omega),\|.\|_0)$ are Hilbert spaces.
\begin{theorem}\label{emb1}(Embedding Theorem)(\cite[Theorem 2.6]{FL3} and \cite[Proposition 3.2]{KAL}) Let $\Omega\subset \R^{k+m}$ be an open set. The embedding 
$$H^{1,\lambda}_0(\Omega)\hookrightarrow L^q(\Omega)$$ is continuous for every $q\in[1,2^*_\lambda]$ and compact for $q\in[1,2^*_\lambda)$, where $2^*_\lambda=\frac{2Q}{Q-2},\; Q=k+(\lambda+1)m$. 
\end{theorem}

\begin{theorem}(Stampacchia-Kinderlehrer \cite[\text{ lemma B.1}] {KDS})\label{stam}
 Let $\phi:[k_0,\infty)\to \R$ be a nonnegative and nonincreasing such that for $ k_0\leq k\leq h$,$$\phi(h)\leq [C/(h-k)^\alpha]|\phi(k)|^\beta$$ where $C,\alpha,\beta$ are positive constant with $\beta>1$. Then $$\phi(k_0+d)=0$$ where $d^\alpha=C2^{\frac{\alpha\beta}{\beta-1}}|\phi(k_0)|^{(\beta-1)}$ 
\end{theorem}
Now we will prove the Strong Maximum Principle for super-solutions of $-\Delta_\lambda u=0$. In this proof, we denote $\rho$ and $S_\rho$, which are defined in \cite[Definition 2.6]{FL1}. The constants $a, c_1$ are introduced in \cite[Theorem 4.3]{FL1}. Also, $c$ and $\epsilon_0$ are defined in \cite[Proposition 2.9]{FL1}.
 \begin{theorem}\label{smp}(Strong Maximum Principle) Let $\Omega\subset \R^{1+m}$ be a $\Lambda-$connected, bounded open set and $\lambda\geq 1$. Let u be a nonnegative 
(not identically zero) function in $H^{1,\lambda}_0(\Omega)$ such that $u$ is a super solution of $-\Delta_\lambda u=0$, i.e., for every nonnegative $v\in H^{1,\lambda}_0(\Omega)$,    $$\int_\Omega \langle A\nabla^*u,\nabla^*v\rangle dX\geq 0.$$ If there exist a ball $B_r(x_0)\Subset\Omega$ with $\inf_{B_r(x_0)}u=0$ then $u$ is identically zero in $\Omega$.
 \end{theorem}
\begin{proof}
Let $n_0$ be a natural number such that $n_0^{\epsilon_0}>2c_1$. We can choose $r>0$ such that $B(X_0,n_0r)\Subset \Omega$, $\inf_{B_r(X_0)}u=0$ and $S_\rho(X,ac(n_0r)^{\epsilon_0})\subset \Omega$. By using (\cite[\text{Proposition 2.9}]{FL1}) and (\cite[\text{Theorem 2.7}]{FL1}) we have $$B(X_0,r)\subset B(X_0,n_0r)\subset S_d(X_0,c(n_0r)^{\epsilon_0})\subset S_\rho(X_0,ac(n_0r)^{\epsilon_0})\subset\Omega$$       
Put $R=\frac{ac(n_0r)^{\epsilon_0}}{c_1}$ and by \cite[\text{Theorem 4.3}]{FL1} with $p=1$, we have 
\begin{equation}\label{HR1}
    \inf_{S_\rho(X_0,\frac{R}{2})}u\geq M |S_\rho(X_0,R)|^{-1}\int_{S_\rho(X_0,R)}|u|\ dX.
\end{equation}
By using (\cite[\text{Proposition 2.9}]{FL1}) and (\cite[\text{Theorem 2.7}]{FL1}) we easily can show that $B(X_0,r)\subset S_\rho(X_0,\frac{R}{2})$. Hence, $\inf_{S_\rho(X_0,\frac{R}{2})}u=0$. By (\ref{HR1}) we have $u=0$ a.e. in $S_\rho(X_0,R)$ and hence, in $B(X_0,r)$. Let $Y\in\Omega$ and $r_0=r$. Since $\Omega$ is a bounded domain, we can find a finite collection of balls $\{B(X_i,r_i)\}_{i=0}^{i=k}$ such that $B(X_i,n_0r_i)\Subset \Omega$, $S_\rho(X_i,ac(n_0r_i)^{\epsilon_0})\subset \Omega$, $B(X_{i-1},r_{i-1})\cap B(X_{i},r_{i})\neq\emptyset$ for $i=1,2...k$ and $Y\in B(X_k,r_k)$. We can use the previous process to show that $u=0$ a.e. in $B(X_1,r_1)$. Iterating we have $u=0$ a.e. in $B(X_k,r_k)$. Hence, $u=0$ a.e.in $\Omega$.                                                             
\end{proof}
Now we are ready to define the notion of solution of (\ref{maineq}).
\begin{defn}\label{def2}
 A function $u\in H^{1,\lambda}_{loc}(\Omega)$ is said to be a weak solution of (\ref{maineq}) if  for every $\Omega'\Subset\Omega$, there exists a positive constant $C(\Omega')$ such that 
 \begin{align*}
     u\geq C(\Omega')>0 \text{ a.e in $\Omega'$},
 \end{align*}
 \begin{align*}
 \int_\Omega \langle A\nabla^*u,\nabla v\rangle\,dX=\int_\Omega \frac{fv}{u^\nu}\,dX\;\mbox{ for all  $v\in C^1_0(\Omega)$}
 \end{align*}
and 
 \begin{itemize}
     \item if $\nu\leq 1$ then $u\in  H^{1,\lambda}_0(\Omega)$.
     \item if $\nu>1$ then $u^{\frac{\nu+1}{2}}\in  H^{1,\lambda}_0(\Omega)$.
 \end{itemize} 
 \end{defn}

\begin{center}
    \section{ EXISTENCE AND REGULARITY RESULTS}\label{sec3}
\end{center}
Henceforth, we will assume $N=1+m$, and  $\Omega\subset\R^N$ is a $\Lambda-$connected, bounded open set. We will also assume $f$ is a nonnegative (not identically zero) function and $\lambda\geq 1$. Our main results are the following:

\subsection{The case $\nu=1$}
\begin{theorem}\label{Th1}
Let $\nu=1$ and $f\in L^1(\Omega)$. Then the Dirichlet boundary value problem (\ref{maineq}) has a unique solution in the sense of definition (\ref{def2}).   
\end{theorem}
\begin{theorem}\label{Th2}
  Let $\nu=1$ and $f\in L^r(\Omega),r\geq 1$. Then the solution given by Theorem \ref{Th1} satisfies the following 
  \begin{enumerate}[label=(\roman*)]
      \item If $r>\frac{Q}{2}$ then $u\in L^\infty(\Omega)$.
      \item If $1\leq r<\frac{Q}{2}$ then 
      $u\in L^s(\Omega)$.
  \end{enumerate}
  where $Q=(m+1)+\lambda m$ and $s=\frac{2Qr}{Q-2r}$.
 \end{theorem}
\subsection{The case $\nu>1$}
\begin{theorem}\label{Th3}
 Let $\nu>1$ and $f\in L^1(\Omega)$. Then there exists $u\in H^{1,\lambda}_\text{loc}(\Omega)$ which satisfies equation (\ref{maineq}) in sense of definition (\ref{def2}).
\end{theorem}
\begin{theorem}\label{Th4}
  Let $\nu>1$ and $f\in L^r(\Omega),\; r\geq1$. Then the solution $u$ of (\ref{maineq}) given by the above theorem is such that 
  \begin{enumerate}[label=(\roman*)]
      \item If $r>\frac{Q}{2}$ then $u\in L^\infty(\Omega).$
      \item If $1\leq r<\frac{Q}{2}$ then $u\in L^s(\Omega).$
  \end{enumerate}
  where $s=\frac{Qr(\nu+1)}{(Q-2r)}$ and $Q=(m+1)+\lambda m.$
 \end{theorem}
 \subsection{The case $\nu<1$}
 \begin{theorem}\label{Th5}
Let $\nu<1$ and $f\in L^r(\Omega),r=(\frac{2^*_\lambda}{1-\lambda})'$. Then (\ref{maineq}) has a unique solution in $H^{1,\lambda}_0(\Omega)$.
\end{theorem}
 
\begin{theorem}\label{Th6}
Let $\nu<1$ and $f\in L^r(\Omega),\; r\geq (\frac{2^*_\lambda}{1-\nu})'$. Then the solution $u$ of (\ref{maineq}) given by the above theorem is such that
\begin{enumerate}[label=(\roman*)]
    \item If $r>\frac{Q}{2}$ then $u\in L^\infty(\Omega).$\\
    \item If $(\frac{2^*_\lambda}{1-\nu})'\leq r<\frac{Q}{2}$ then $u\in L^s(\Omega).$
\end{enumerate}
where $s=\frac{Qr(\nu+1)}{(Q-2r)},\; Q=(m+1)+\lambda m$ and $r'$ denotes the H\"{o}lder conjugate of $r$.
\end{theorem}

\begin{theorem}\label{Th7}
Let $\nu<1$ and $f\in L^r(\Omega)$ for some $1\leq r<\frac{2Q}{(Q+2)+\nu(Q-2)}$. Then there exists $u\in W^{1,\lambda,q}_0(\Omega)$ which is a solution of (\ref{maineq}) in the sense 
\begin{align*}
    \int_\Omega\langle A\nabla^*u,\nabla v\rangle dX=\int_\Omega \frac{fv}{u^\nu}\;dX \mbox{ for all $v\in C^1_0(\Omega)$}
\end{align*}
where $q=\frac{Qr(\nu+1)}{Q-r(1-\nu)}$.
\end{theorem}

 \begin{center}
      \section{APPROXIMATION OF THE EQUATION (\ref{maineq})}\label{sec4}
 \end{center}

Let $f$ be a nonnegative (not identically zero) measurable function and $n\in N$. Let us consider the equation 
\begin{align}\label{equ}
    -\Delta_\lambda u_n&=\frac{f_n}{(u_n+\frac{1}{n})^\nu}
     \text{ in }\Omega\\
   u=0 &\text{ on } \partial\Omega\nonumber
\end{align}
where $f_n:=\min\{f,n\}$.\\
 \begin{lem}
 Equation (\ref{equ}) has a unique solution $u_n\in H^{1,\lambda}_0(\Omega)\cap L^\infty(\Omega)$.
 \end{lem}
 \begin{proof}
 Let $w\in L^2(\Omega)$ be a fixed element. Now consider the equation 
 \begin{align}\label{P2}
    -\Delta_\lambda u&=g_n\text{ in }\Omega\\
    u&=0 \text{ on } \partial\Omega\nonumber
\end{align}
where $g_n=\frac{f_n}{(|w|+\frac{1}{n})^\nu}$. Since $|g_n(x)|\leq n^{\nu+1}$ one has $g_n\in L^2(\Omega)$.
By \cite[Theorem 4.4]{MP}, we can say equation (\ref{P2}) has a unique solution $u_w\in H^{1,\lambda}_0(\Omega)$ and the map $T:L^2(\Omega)\to H^{1,\lambda}_0(\Omega)$ such that $T(w)=u_w$ is continuous. By Theorem \ref{emb1}, we have the compact embedding $$H^{1,\lambda}_0(\Omega)\hookrightarrow L^2(\Omega).$$
Hence, the $T:L^2(\Omega)\to L^2(\Omega)$ is continuous as well as compact.\\
Let $S=\{w\in L^2(\Omega) : w=\lambda Tw \;\text{for some}\; 0\leq \lambda\leq 1\}$.\\
\textbf{ Claim:} The set $S$ is bounded. \\
 Let $w\in S$. By the Poincar\'{e} inequality (see \cite[Theorem 2.1]{MP}), there exists a constant $C>0$ such that, \\
 \begin{align*}
     \|u_w\|_{L^2(\Omega)}^2&\leq C\int_\Omega \langle A\nabla^*u_w,\nabla^*u_w\rangle \,dX
     =C\int_\Omega g_n(x)u_w\,dX
     \leq Cn^{\nu+1}\int_\Omega u_w\,dX
     \leq Cn^{\nu+1}|\Omega|^\frac{1}{2} \|u_w\|_{L^2(\Omega)}
\end{align*}
Hence, we have
\begin{align*}
    \|u_w\|_{L^2(\Omega)}&\leq Cn^{\nu+1}|\Omega|^\frac{1}{2}
\end{align*}
 where $C>0$ is a independent of $w$.  This proves $S$ is bounded. Hence by Schaefer's fixed point theorem, there exists $u_n\in H^{1,\lambda}_0(\Omega)$ such that
 \begin{align}
    -\Delta_\lambda u_n&=\frac{f_n}{(|u_n|+\frac{1}{n})^\nu}
     \text{ in }\Omega\\
   u=0 &\text{ on } \partial\Omega\nonumber
\end{align}
By Weak Maximum Principle (see \cite[\text{Theorem} 4.4]{MP}), we have $u_n\geq 0$ in $\Omega$. So $u_n$ is a solution of (\ref{equ}). Hence, 
\begin{equation}\label{weakd}
    \int_\Omega \langle A\nabla^*u_n,\nabla v\rangle dX=\int_\Omega \frac{f_nv}{(u_n+\frac{1}{n})^\nu}dX 
     \text{ for every }v\in C^1_0(\Omega)
\end{equation}
Now, we want to prove $u_n\in L^\infty(\Omega)$.\\
Let $k>1$ and define $ S(k)=\{x\in \Omega : u_n(x)\geq k\}$. We can treat the function \begin{equation*}
     v(x)= \begin{cases}
             u_n(x)-k & x\in S(k)\\
             o & \text{otherwise}
             \end{cases}
\end{equation*}
as a function in $C^1_0(\Omega)$. By putting $v$ in (\ref{weakd}), we obtain
\begin{align}\label{weakin}
    \int_{S(k)} \langle A\nabla^*v,\nabla^* v\rangle\,dX&=\int_{S(k)} \frac{f_nv}{(v+k+\frac{1}{n})^\nu}\,dX\nonumber
    \leq n^{\nu+1} \int_{S(k)}v\,dX\nonumber
    \leq n^{\nu+1} \|v\|_{L^{2^*_\lambda}(\Omega)} |S(k)|^{1-\frac{1}{2^*_\lambda}}
\end{align}
Here, $2^*_\lambda=\frac{2Q}{Q-2}$  and $Q=(m+1)+\lambda m$.
Now, by Theorem \ref{emb1} there exists $C>0$ such that 
\begin{align*}
    \|v\|_{L^{2^*_\lambda}(\Omega)}^2&\leq  C\int_{\Omega} \langle A\nabla^*v,\nabla^* v\rangle\,dX \nonumber
    =C\int_{S(k)} \langle A\nabla^*v,\nabla^* v\rangle\,dX \nonumber
    \leq Cn^{\nu+1}\|v\|_{L^{2^*_\lambda}(\Omega)} |S(k)|^{1-\frac{1}{2^*_\lambda}}.
    \nonumber
\end{align*}
 We have
\begin{equation}\label{S1}
     \|v\|_{L^{2^*_\lambda}(\Omega)}\leq  Cn^{\nu+1}|S(k)|^{1-\frac{1}{2^*_\lambda}}
\end{equation}
Assume $1<k<h$ and using Inequality (\ref{S1}) we get
\begin{align*}
    |S(h)|^\frac{1}{2^*_\lambda}(h-k)&=(\int_{S(h)}(h-k)^{2^*_\lambda}\,dX)^\frac{1}{2^*_\lambda}
    \leq (\int_{S(k)}(v(x))^{2^*_\lambda}\,dX)^\frac{1}{2^*_\lambda}
    \leq \|v\|_{L^{2^*_\lambda}(\Omega)}
    \leq Cn^{\nu+1}|S(k)|^{1-\frac{1}{2^*_\lambda}}
\end{align*}
The above two inequalities implies
\begin{align*}
    |S(h)|\leq (\frac{Cn^{\nu+1}}{(h-k)})^{2^*_\lambda})|S(k)|^{2^*_\lambda-1}
\end{align*}
Let $d^{2^*_\lambda}=(Cn^{\nu+1})^{2^*_\lambda})2^\frac{2^*_\lambda(2^*_\lambda-1)}{2^*_\lambda-2} |S(1)|^{2^*_\lambda-2}$ then by the Theorem \ref{stam}, we get $|S(1+d)|=0$. Hence, $u_n(x)\leq 1+d$ a.e in $\Omega$. We get a positive constant $C(n)$ such that $u_n\leq C(n)$ a.e in $\Omega$. Consequently, $u_n\in L^\infty(\Omega)$.\\
Let $u_n$ and $v_n$ be two solutions of (\ref{equ}). The function $w=(u_n-v_n)^+\in H^{1,\lambda}_0(\Omega)$ can be considered as a test function. It is clear that
\begin{align}\label{f1}
    [(v_n+\frac{1}{n})^\nu-(u_n+\frac{1}{n})^\nu]w\leq 0
\end{align}
Since $u_n$ and $v_n$ are two solutions of (\ref{equ}) so by putting $w$ in (\ref{weakd}) we get
\begin{align*}
    \int_\Omega \langle A\nabla^*u_n,\nabla^* w\rangle dX=\int_\Omega \frac{f_nw}{(u_n+\frac{1}{n})^\nu}dX \\
 \text{and}   \int_\Omega \langle A\nabla^*v_n,\nabla^* w\rangle dX=\int_\Omega \frac{f_nw}{(v_n+\frac{1}{n})^\nu}dX 
\end{align*}
Therefore,
\begin{align*}
   \int_\Omega \langle A\nabla^*(u_n-v_n),\nabla^* w\rangle\,dX &=\int_\Omega \frac{f_n[(v_n+\frac{1}{n})^\nu-(u_n+\frac{1}{n})^\nu]}{(u_n+\frac{1}{n})^\nu(v_n+\frac{1}{n})^\nu}w\,dX\\
\end{align*}
Using (\ref{f1}) we have 
\begin{align*}
    \int_\Omega \langle A\nabla^*w,\nabla^* w\rangle\,dX\leq 0
\end{align*}
Hence, $w=0$ and so $(u_n-v_n)\leq 0$. By a similar argument, we can prove that $(v_n-u_n)\leq 0$. Consequently, $u_n=v_n$ a.e in $\Omega$. 
 \end{proof}
 \begin{lem}\label{lem1}
 Let for each $n\in \N$, $u_n$ be the solution of (\ref{equ}). Then the sequence $\{u_n\}$ is an increasing sequence  and for each $\Omega'\Subset \Omega$, there exists a constant $C(\Omega')>0$ such that
 $$u_n(x)\geq C(\Omega')>0\;\mbox{ a.e}\; x\in\Omega'\;\mbox{ and for all}\; n\in\N$$
 \end{lem}
 \begin{proof}
 Let $n\in \N$ be fixed. Define $w=(u_n-u_{n+1})^+$. It is clear that $$[(u_{n+1}+\frac{1}{n+1})^\nu-(u_n+\frac{1}{n})^\nu]w\leq0.$$ $w$ can be considered as a test function.
 Arguing as in the proof of the previous theorem, we obtain $w=0$. Hence, $u_n-u_{n+1}\leq 0$ $\implies u_n\leq u_{n+1}$ a.e in $\Omega$ and for all $n\in\N$. Since $f$ is not identically zero so $f_i$ is not identically zero for some $i\in N$. Without loss of generality, we may assume that $f_1$ is not identically zero.\\
 Consider the equation
 \begin{align}
    -\Delta_\lambda u_1&=\frac{f_1}{(u_1+1)^\nu}\text{ in}\;\Omega\\
   u_1&=0 \text{ on } \partial\Omega\nonumber
\end{align}
 Since $f_1$ is not identically zero so $u_1$ is not identically zero. So by Theorem \ref{smp}, we have $u_1>0$ in $\Omega$. Hence, for every compact set $\Omega'\Subset \Omega$, there exists a constant $C(\Omega')>0$ such that $u_1\geq C(\Omega')$ a.e. in $\Omega'$. Monotonicity of the sequence implies that for every $n\in N$,
 \begin{align*}
     u_n\geq C(\Omega').
 \end{align*}
 \end{proof}
 
 \begin{center}
     \section{ A FEW AUXILIARY RESULTS}\label{sec5}
 \end{center}
 We start this section with the proof of a priori estimates on $u_n$.
 \begin{lem}\label{L1}
 Let $u_n$ be the solution of equation (\ref{equ}) with $\nu=1$ and assume $f\in L^1(\Omega)$ is a nonnegative function (not identically zero). Then the sequence $\{u_n\}$ is bounded in $H^{1,\lambda}_0(\Omega)$.
 \end{lem}
 \begin{proof}
 Since $u_n\in H^{1,\lambda}_0(\Omega)$ is a solution of (\ref{equ}) so from (\ref{weakd}) we obtain \\
 \begin{align*}
     \int_\Omega \langle A\nabla^*u_n,\nabla^* u_n\rangle\,dX&=\int_\Omega \frac{f_nu_n}{(u_n+\frac{1}{n})}dX
     \leq \int_\Omega fdX
     =\|f\|_{L^1(\Omega)}
 \end{align*}
 Hence, $\{u_n\}$ is bounded in $H^{1,\lambda}_0(\Omega)$.
 \end{proof}
 \begin{lem}\label{L2}
 Let $u_n$ be the solution of the equation (\ref{equ}) with $\nu>1$ and $f\in L^1(\Omega)$ is a nonnegative function (not identically zero). Then $\{u_n^{\frac{\nu+1}{2}}\}$ is bounded in $H^{1,\lambda}_0(\Omega)$ and $\{u_n\}$ is bounded in $H^{1,\lambda}_\text{loc}(\Omega)$ and in $L^s(\Omega)$, where $s=\frac{(\nu+1)Q}{(Q-2)}$.
  \end{lem}
 \begin{proof}
 Since $\nu>1$ and $u_n\in H^{1,\lambda}_0(\Omega) $ so by putting $v=u_n^\nu$ in (\ref{weakd}) we have,
 \begin{align*}
     \int_\Omega\langle A\nabla^*u_n,\nabla^*u_n^\nu\rangle dX&=\int_\Omega \frac{f_nu_n^\nu}{(u_n+\frac{1}{n})^\nu}dX
      \leq\int_\Omega fdX.
 \end{align*}
 Now,
 \begin{align}\label{f2}
      \int_\Omega \langle A\nabla^*u_n^{\frac{\nu+1}{2}},\nabla^*u_n^{\frac{\nu+1}{2}}\rangle dX=\frac{(\nu+1)^2}{4\nu}\int_\Omega \nu u_n^{\nu-1}\langle A\nabla^*u_n,\nabla^*u_n\rangle dX&=\frac{(\nu+1)^2}{4\nu}\int_\Omega\langle A\nabla^*u_n,\nabla^*u_n^\nu\rangle dX\nonumber\\
     \leq \frac{(\nu+1)^2}{4\nu}\int_\Omega fdX.
 \end{align}
     Hence, $\{u_n^\frac{\nu+1}{2}\}$ is bounded in $H^{1,\lambda}_0(\Omega)$. By Theorem \ref{emb1}, there exists a constant $C>0$ such that 
     \begin{align*}
         \|u_n^{\frac{\nu+1}{2}}\|_{L^{2^*_\lambda}(\Omega)}\leq C\|u_n^{\frac{\nu+1}{2}}\|_{H^{1,\lambda}_0(\Omega)}
     \end{align*}
     By using (\ref{f2}), we have 
     \begin{align*}
          (\int_\Omega u_n^{2^*_\lambda\frac{(\nu+1)}{2}} dX)^\frac{2}{2^*_\lambda}\leq C\frac{(\nu+1)^2}{4\nu}\|f\|_{L^1(\Omega)}
     \end{align*}
     Since $s={2^*_\lambda\frac{(\nu+1)}{2}}$ so
\begin{align*}
    \int_\Omega u_n^s dX\leq (C\frac{(\nu+1)^2}{4\nu}\|f\|_{L^1(\Omega)})^\frac{2^*_\lambda}{2}
\end{align*}
 Hence, $\{u_n\}$ is bounded in $L^s(\Omega).$ To prove $\{u_n\}$ is bounded in $H^{1,\lambda}_\text{loc}(\Omega)$, let $\Omega'\Subset\Omega$ and $\eta\in C^\infty_0(\Omega)$ such that $0\leq\eta\leq1$ and 
     $\eta=1$ in $\Omega'$.
     It is a test function as $u_n\eta^2\in H^{1,\lambda}_0(\Omega)$. By Lemma \ref{lem1}, there exists a constant $C>0$ such that $u_n\geq C$ a.e in \text{supp}($\eta$). Put $v=u_n\eta^2$ in (\ref{weakd}) we have
     \begin{align}\label{w0}
       \int_\Omega\langle A\nabla^*u_n,\nabla^*(u_n\eta^2)\rangle dX=\int_\Omega \frac{f_nu_n\eta^2}{(u_n+\frac{1}{n})^\nu}dX
       \end{align}
       Also,
       \begin{align}\label{f5}
       \int_\Omega\langle A\nabla^*u_n,\nabla^*(u_n\eta^2)\rangle dX=\int_\Omega \{\eta^2\langle A\nabla^*u_n,\nabla^*u_n\rangle+2\eta u_n\langle A\nabla^*u_n,\nabla\eta\rangle\}
     \end{align}
     From (\ref{w0}) and (\ref{f5}) we get
     \begin{align}\label{f6}
       \int_\Omega \eta^2\langle A\nabla^*u_n,\nabla^*u_n\rangle dX= \int_\Omega \frac{f_n\eta^2}{C^{(\nu-1)}}dX-\int_\Omega2\eta u_n\langle A\nabla^*u_n,\nabla\eta\rangle dX
     \end{align}
     Choose  $\epsilon>0$ and use Young's inequality; one has
     \begin{align}\label{w8}
         |\int_\Omega2\eta u_n\langle A\nabla^*u_n,\nabla\eta\rangle dX|&\leq \int_\Omega2|\langle\eta\sqrt{A}\nabla^*u_n,u_n\sqrt{A}\nabla\eta\rangle|dX\nonumber\\
         &\leq \frac{1}{\epsilon}\int_\Omega \eta^2 |\sqrt{A}\nabla^*u_n|^2dX+\epsilon \int_\Omega u_n^2 |\sqrt{A}\nabla\eta|^2dX,
 \end{align}
 
 Put $\epsilon=2$ then we get
 \begin{align}\label{w9}
    |\int_\Omega2\eta u_n\langle A\nabla^*u_n,\nabla\eta\rangle dX|&\leq\frac{1}{2}\int_\Omega \eta^2 |\sqrt{A}\nabla^*u_n|^2dX+2\int_\Omega u_n^2 |\sqrt{A}\nabla\eta|^2dX \nonumber\\
    &=\frac{1}{2}\int_\Omega \eta^2 \langle A\nabla^*u_n,\nabla^*u_n\rangle dX+2\int_\Omega u_n^2 \langle A\nabla \eta,\nabla \eta\rangle dX 
 \end{align}
 Using (\ref{f6}) and (\ref{w9}), we have 
 \begin{align*}
     \int_\Omega \eta^2\langle A\nabla^*u_n,\nabla^*u_n\rangle dX&\leq 2\int_\Omega \frac{f\eta^2}{C^{(\nu-1)}}dX+4\int_\Omega u_n^2 \langle A\nabla \eta,\nabla \eta\rangle dX\\
     &\leq \frac{2\|\eta\|_\infty^2\|f\|_{L^1(\Omega)}}{C^{\nu-1}}+ 4\|\langle A\nabla \eta,\nabla \eta\rangle\|_\infty\int_\Omega u_n^2  dX
 \end{align*}
Since $\{u_n\}$ is bounded in $L^s(\Omega)$ and $s>2$ So $\{u_n\}$ is bounded in $L^2(\Omega)$.
\begin{align*}
     \int_\Omega \eta^2\langle A\nabla^*u_n,\nabla^*u_n\rangle dX&\leq \frac{2\|\eta\|_\infty^2\|f\|_{L^1(\Omega)}}{C^{\nu-1}}+4\|\langle A\nabla \eta,\nabla \eta\rangle\|_\infty\int_\Omega u_n^2  dX\\
     &\leq C(f,\eta)
\end{align*}
 Now, $$\int_{\Omega'}\langle A\nabla^*u_n,\nabla^*u_n\rangle dX\leq \int_\Omega \eta^2\langle A\nabla^*u_n,\nabla^*u_n\rangle dX\leq C(f,\eta)$$
 Hence, $\{u_n\}$ is bounded in $H^{1,\lambda}_\text{loc}(\Omega)$.
\end{proof}

\begin{lem}\label{L3}
Let $u_n$ be the solution of (\ref{equ}) with $\nu<1$ and $f\in L^r$, $r=(\frac{2^*_\lambda}{1-\nu})'$ is a nonnegative (not identically zero) function. Then $\{u_n\}$ is bounded in $H^{1,\lambda}_0(\Omega)$.
\end{lem}
\begin{proof}
 Since $r=(\frac{2^*_\lambda}{1-\nu})'$, we can choose $v=u_n$ in (\ref{weakd}) and using H\"{o}lder inequality, one has 
 \begin{align}\label{f3}
     \int_\Omega\langle A\nabla^*u_n,\nabla^*u_n\rangle dX=\int_\Omega \frac{f_nu_n}{(u_n+\frac{1}{n})^\nu}\leq \int_\Omega fu_n^{1-\nu} dX&\leq \|f\|_{L^r(\Omega)}(\int_\Omega u_n^{(1-\nu)r'} dX)^\frac{1}{r'}\nonumber\\
     &\leq \|f\|_{L^r(\Omega)}(\int_\Omega u_n^{2^*_\lambda} dX)^\frac{1-\nu}{2^*_\lambda}.
 \end{align}
 
By Theorem \ref{emb1} and using the above inequality, we get 

\begin{align}\label{w6}
    \int_\Omega u_n^{2^*_\lambda} dX&\leq C(\int_\Omega\langle A\nabla^*u_n,\nabla^*u_n\rangle dX)^\frac{2^*_\lambda}{2}\leq C (\|f\|_{L^r(\Omega)}(\int_\Omega u_n^{2^*_\lambda} dX)^\frac{1-\nu}{2^*_\lambda})^\frac{2^*_\lambda}{2}.
 \end{align}
 So we have 
 \begin{align}\label{P1}
    \int_\Omega u_n^{2^*_\lambda} dX\leq C\|f\|_{L^r(\Omega)}^\frac{2^*_\lambda}{1+\nu}. 
 \end{align}
 Hence, $\{u_n\}$ is bounded $L^{2^*_\lambda}(\Omega)$. Using (\ref{f3}) and (\ref{P1}), we can conclude $\|u_n\|_{H^{1,\lambda}_0(\Omega)}\leq C\|f\|_{L^r(\Omega)}^\frac{1}{1+\nu}$
where $C$ is independent of $n$ . Hence, $\{u_n\}$ is bounded in $H^{1,\lambda}_0(\Omega)$. 
\end{proof}

\begin{center}
    \section{PROOF OF MAIN RESULTS}\label{sec6}
\end{center}

 %\begin{theorem}\label{Th1}
  %Let $\nu=1$ and $f$ be a nonnegative (not identically zero) in $L^1(\Omega)$. Then (\ref{maineq}) has a unique solution in $H^{1,\lambda}_0(\Omega)$ in the sense of definition(\ref{def2}).
 %end{theorem}

 %Since $\{u_n\}$ is a monotonically increasing sequence of functions so we can define a %function $u$ as the pointwise limit of $u_n$. Now, we will prove our main results in %this section.\\
 \subsection{The case $\nu=1$}
 \textbf{Proof of Theorem \ref{Th1}:}
 \begin{proof}
   Consider the above sequence $\{u_n\}$ and define $u$ as the pointwise limit of the sequence ${\{u_n\}}$. Since $H^{1,\lambda}_0(\Omega)$ is Hilbert space and $\{u_n\}$ is bounded in $H^{1,\lambda}_0(\Omega)$
   so it admits a weakly convergent subsequence. Assume $u_n$ weakly converges to $v$ in $H^{1,\lambda}_0(\Omega)$ and hence $u_n$ converges to $v$ in $L^2(\Omega)$. So $\{u_n\}$ has a subsequence that converges to $v$ pointwise. Consequently, $u=v$. So we may assume that the sequence  $\{u_n\}$  weakly converges to $u$ in $H^{1,\lambda}_0(\Omega)$. Choose $v'\in C^1_0(\Omega)$. By Lemma \ref{lem1}, there exists $C>0$ such that $u\geq u_n\geq C$ a.e in supp(v') and for all $n\in \N$. So $$|\frac{f_nv'}{(u_n+\frac{1}{n})}|\leq \frac{\|v'\|_\infty|f|}{C}\;\mbox{ for all} n\in\N$$
   By Dominated Convergence Theorem, we have 
   \begin{align}\label{DCT}
       \lim_{n\to \infty}\int_\Omega \frac{f_nv'}{(u_n+\frac{1}{n})}dX= \int_\Omega \lim_{n\to \infty} \frac{f_nv'}{(u_n+\frac{1}{n})}dX
       =\int_\Omega \frac{fv'}{u}dX.
   \end{align}
   As $u_n$ is a solution of (\ref{equ}) so from (\ref{weakd}) we get,
 \begin{align*}
     &\int_\Omega \langle A\nabla^*u_n,\nabla v'\rangle dX=\int_\Omega \frac {f_nv'}{(u_n+\frac{1}n)} dX
 \end{align*}
 Take $n\to \infty$ and use (\ref{DCT}) we obtain,
 \begin{align*}
      \int_\Omega \langle A\nabla^*u,\nabla v'\rangle dX=\int_\Omega \frac {fv'}{u} dX 
 \end{align*}
 Hence, $u\in H^{1,\lambda}_0(\Omega)$ is a solution of (\ref{maineq}).\\
 Let $u$ and $v$ be two solutions of (\ref{maineq}). The function $w=(u-v)^+\in H^{1,\lambda}_0(\Omega)$ can be considered as a test function. 
Since $u_n$ and $v_n$ are two solutions of (\ref{maineq}) so we have
\begin{align*}
    \int_\Omega \langle A\nabla^*u,\nabla^* w\rangle dX&=\int_\Omega \frac{fw}{u}dX \\
& \text{and}\\  \int_\Omega \langle A\nabla^*v,\nabla^* w\rangle dX&=\int_\Omega \frac{fw}{v}dX 
\end{align*}
By subtracting one from the other, we get
\begin{align*}
   \int_\Omega \langle A\nabla^*(u-v),\nabla^* w\rangle\,dX &=\int_\Omega \frac{f(v-u)}{uv}w dX\leq 0.
\end{align*}
Which ensures us 
\begin{align*}
    \int_\Omega \langle A\nabla^*w,\nabla^* w\rangle\,dX\leq 0.
\end{align*}
Hence, $w=0$ and so $(u-v)\leq 0$. By interchanging the role of $u$ and $v$, we get $(v-u)\leq 0$. Consequently, $u=v$ a.e in $\Omega$.
\end{proof}

 %\begin{theorem}\label{thm1}
  %Let $u$ be a solution of (\ref{maineq}) when $\nu=1$. If $f\in L^r,r\geq1$ is a not identically zero function, then the following hold
  %\end{theorem}
  \textbf{Proof of Theorem \ref{Th2}:}
 \begin{proof}
 $(i)$ Let $k>1$ and define $ S(k)=\{x\in \Omega : u_n(x)\geq k\}$. We can treat the function
 \begin{equation*}
     v(x)= \begin{cases}
             u_n(x)-k & x\in S(k)\\
             o & \text{otherwise}
             \end{cases}
\end{equation*}
as a function in $C^1_0(\Omega)$. So by $(5)$ we have
\begin{align}\label{w1}
    \int_{S(k)} \langle A\nabla^*v,\nabla^* v\rangle\,dX&=\int_{S(k)} \frac{f_nv}{(v+k+\frac{1}{n})}dX
    \leq  \int_{S(k)} fv\,dX
    \leq \|f\|_{L^r(\Omega)} \|v\|_{L^{2^*_\lambda}(\Omega)} |S(k)|^{1-\frac{1}{2^*_\lambda}-\frac{1}{r}} 
\end{align}
where $2^*_\lambda=\frac{2Q}{Q-2}$.
By Theorem \ref{emb1}, there exists $C>0$ such that 
\begin{align}\label{w2}
    \|v\|_{L^{2^*_\lambda}(\Omega)}^2&\leq  C\int_{\Omega} \langle A\nabla^*v,\nabla^* v\rangle dX 
   =C\int_{S(k)} \langle A\nabla^*v,\nabla^* v\rangle\,dX 
   \leq C\|f\|_{L^r(\Omega)} \|v\|_{L^{2^*_\lambda}(\Omega)} |S(k)|^{1-\frac{1}{2^*_\lambda}-\frac{1}{r}}
\end{align}
The last inequality follows from (\ref{w1}). Inequality (\ref{w2}) ensures us 
$$\|v\|_{L^{2^*_\lambda}(\Omega)}\leq  C\|f\|_{L^r(\Omega)} |S(k)|^{1-\frac{1}{2^*_\lambda}-\frac{1}{r}}$$
 Assume $1<k<h$. Using last inequality, we obtain
\begin{align*}
    |S(h)|^\frac{1}{2^*_\lambda}(h-k)&=(\int_{S(h)}(h-k)^{2^*_\lambda}\,dX)^\frac{1}{2^*_\lambda}
    \leq (\int_{S(k)}(v(x))^{2^*_\lambda}\,dX)^\frac{1}{2^*_\lambda}
    \leq \|v\|_{L^{2^*_\lambda}(\Omega)}
    \leq C\|f\|_{L^r(\Omega)} |S(k)|^{1-\frac{1}{2^*_\lambda}-\frac{1}{r}}
\end{align*}
So, $$|S(h)|\leq (\frac{C\|f\|_{L^r(\Omega)}}{(h-k)})^{2^*_\lambda}  |S(k)|^{{2^*_\lambda}(1-\frac{1}{2^*_\lambda}-\frac{1}{r})} \\$$
As $r>\frac{Q}{2}$ we have, $2^*_\lambda({1-\frac{1}{2^*_\lambda}-\frac{1}{r}})>1$.
Let $$d^{2^*_\lambda}=(C\|f\|_{L^r(\Omega)})^{2^*_\lambda}2^\frac{(2^*_\lambda)^2(1-\frac{1}{2^*_\lambda}-\frac{1}{r})}{[{2^*_\lambda}(1-\frac{1}{(2^*_\lambda}-\frac{1}{r})-1]} |S(1)|^{{2^*_\lambda}(1-\frac{1}{2^*_\lambda}-\frac{1}{r})-2}$$
By Theorem \ref{stam} we have $|S(1+d)|=0$. Hence, $u_n(x)\leq 1+d$ a.e in $\Omega$ . We get a positive constant $C$ independent of $n$ such that $u_n\leq C\|f\|_{L^r(\Omega)}$ a.e in $\Omega$ for all $n\in\N$. Hence, $\|u\|_{L^\infty(\Omega)}\leq C\|f\|_{L^r(\Omega)}$\\ 

$(ii)$
If $r=1$ then $s=2^*_\lambda$. Since $u\in H^{1,\lambda}_0(\Omega)$ so by Theorem \ref{emb1}, we have $u\in L^s(\Omega)$.\\
If $1<r<\frac{Q}{2}$ . Choose $\delta>1$ (to be determined later). Consider the function $w=u^{2\delta-1}$. By the density argument, $w$ can be treated as a test function. Put $w$ in (\ref{weakd}), we have
\begin{align*}
    \int_\Omega (2\delta-1)u_n^{(2\delta-2)}\langle A\nabla^*u_n,\nabla^*u_n\rangle dX=\int_\Omega \frac{f_nw}{u_n+\frac{1}{n}}dX
   \leq \int_\Omega fu_n^{2\delta-2 } dX
\end{align*} 
By using H\"{o}lder inequality on the RHS of the above inequality, we get
\begin{equation}\label{w3}
   \int_\Omega\langle A\nabla^*u_n^\delta,\nabla^*u_n^\delta\rangle dX =\int_\Omega \delta^2 u_n^{(2\delta-2)}\langle A\nabla^*u_n,\nabla^*u_n\rangle dX\leq \frac{\delta^2}{(2\delta-1)}\|f\|_{L^r(\Omega)}(\int_\Omega u_n^{(2\delta-2)r'} dX)^\frac{1}{r'} 
\end{equation}
where $\frac{1}{r}+\frac{1}{r'}=1$. By Theorem \ref{emb1}, we have \\
\begin{align}\label{w4}
    \int_\Omega u_n^{2^*_\lambda\delta}
    &\leq C(\int_\Omega \langle A\nabla^*{u_n^\delta},\nabla^*{u_n^\delta}\rangle dX)^\frac{2^*_\lambda}{2}\nonumber\\
    &\leq C\{\frac{\delta^2}{(2\delta-1)}\|f\|_{L^r(\Omega)}(\int_\Omega u_n^{(2\delta-2)r'} dX)^\frac{1}{r'}\}^\frac{2^*_\lambda}{2}, \text{ [by (\ref{w3})]}
\end{align}
 We choose $\delta$ such that $2^*_\lambda\delta=(2\delta-2)r'$ so $\delta=\frac{r(Q-2)}{(Q-2r)}$.
Clearly, $\delta>1$ and $2^*_\lambda\delta=s$ . By using (\ref{w4}), we have 
\begin{align*}
    (\int_\Omega u_n^sdX)^{(1-\frac{2^*_\lambda}{2r'})}\leq C
\end{align*}
 Also, $(1-\frac{2^*_\lambda}{2r'})>0$ as $r<\frac{Q}{2}$. So we get 
$$\int_\Omega u_n^sdX\leq C, \text{   $C>0$ is independent of $n$}.$$
By Dominated Convergence Theorem, we have 
$$\int_\Omega u^sdX\leq C.$$
Hence we are done.
\end{proof}
\subsection{The Case $\nu>1$}
\textbf{Proof of Theorem \ref{Th3}:}
 
%\begin{theorem}
 %Let $\nu>1$ and $f\in L^1(\Omega)$ be a non-negative function (not identically zero ). Then there exists $u\in H^{1,\lambda}_\text{loc}(\Omega)$ which satisfies definition (\ref{def2}) and $u^\frac{\nu+1}{2}\in H^{1,\lambda}_0(\Omega) $.
%\end{theorem}
 \begin{proof}
 Define $u$ as the pointwise limit of $\{u_n\}$.  By Lemma \ref{L2}, $\{u_n\}$ and $\{u_n^{\frac{\nu+1}{2}}\}$ are bounded in $H^{1,\lambda}_{loc}(\Omega)$ and $H^{1,\lambda}_0(\Omega)$ respectively. So by the similar argument as the proof of Theorem \ref{Th1} we can prove $u\in H^{1,\lambda}_{loc}(\Omega)$ and $u^{\frac{\nu+1}{2}}\in H^{1,\lambda}_0(\Omega)$.\\
 Let $v\in C^1_0(\Omega)$ and $\Omega'=\text{supp}(v)$. Without loss of generality we can assume $u_n$ weakly converges to $u$ in $H^{1,\lambda}(\Omega')$. By Lemma \ref{lem1}, there exists $C>0$ such that $u_n(x)\geq C$ a.e $x\in\Omega'$ and for all $n\in \N$. So, $u\geq C>0$ a.e in $\Omega'$. Also, $$|\frac{f_nv}{(u_n+\frac{1}{n})^\nu}|\leq \frac{\|v\|_\infty|f|}{C^\nu},\mbox{ for all} n\in \N$$
 By the Dominated Convergence Theorem, we have 
   \begin{align}\label{DCT1}
       \lim_{n\to \infty}\int_{\Omega'} \frac{f_nv}{(u_n+\frac{1}{n})^\nu}dX&= \int_{\Omega'} \lim_{k\to \infty} \frac{f_nv}{(u_n+\frac{1}{n})^\nu}dX
       =\int_{\Omega'} \frac{fv}{u^\nu}dX.
   \end{align}
   As $u_n$ is a solution of (\ref{equ}) so
 \begin{align*}
     \int_{\Omega'}\langle A\nabla^*u_n,\nabla v\rangle dX=\int_{\Omega'}\frac{f_nv}{(u_n+\frac{1}{n})^\nu} dX
 \end{align*}
 Take $n\to\infty$ and use (\ref{DCT1}), we get
 \begin{align*}
     \int_{\Omega}\langle A\nabla^*u,\nabla v\rangle dX=\int_{\Omega}\frac{fv}{u^\nu} dX
 \end{align*}
 Hence, $u\in H^{1,\lambda}_{loc}(\Omega)$ is a solution of (\ref{maineq}).
 \end{proof}
 
 %\begin{theorem}\label{T2}
  %Let $\nu>1$ and $f\in L^r(\Omega), r\geq1$ then the solution $u$ of (\ref{maineq}) is such that 
  %\begin{enumerate}[label=(\roman*)]
   %   \item If $r>\frac{Q}{2}$ then $u\in L^\infty(\Omega).$
    %  \item If $1\leq r<\frac{Q}{2}$ then $u\in L^s(\Omega)$
  %\end{enumerate}
  %where $s=\frac{Qr(\nu+1)}{(Q-2r)}$ and $Q=(m+1)+\lambda m$
 %\end{theorem}
 \textbf{Proof of Theorem \ref{Th4}:}
\begin{proof}
(i) The same proof of Theorem (\ref{Th2}) will work.\\
(ii) If $r=1$ then $s=\frac{2^*_\lambda(\nu+1)}{2}$. Also, $u^\frac{\nu+1}{2}\in H^{1,\lambda}_0(\Omega)$. By Theorem \ref{emb1}, we have $u\in L^s(\Omega)$.\\
If $1<r<\frac{Q}{2}$. Choose $\delta>\frac{\nu+1}{2}$. By the density argument, $v=u_n^{2\delta-1}$ can be considered a test function. From (\ref{weakd}), we have
$$\int_\Omega\langle A\nabla^*u_n,\nabla^*u_n^{2\delta-1}\rangle\,dX=\int_\Omega \frac{f_nu_n^{2\delta-1}}{(u_n+\frac{1}{n})^\nu}\,dX$$
which gives us
\begin{align}\label{w5}
   \int_\Omega (2\delta-1)u_n^{2\delta-2}\langle A\nabla^*u_n,\nabla^*u_n\rangle dX&\leq \int_\Omega fu_n^{2\delta-\nu-1}dX
    \leq \|f\|_{L^r(\Omega)}(\int_\Omega u_n^{(2\delta-\nu-1)r'}dX)^\frac{1}{r'}
 \end{align}
By Theorem \ref{emb1}, there exists $C>0$ such that 
\begin{align}\label{g1}
    \int_\Omega u_n^{\delta2^*_\lambda} dX&\leq C(\int_\Omega \langle A\nabla^*u_n^\delta,\nabla^*u_n^\delta\rangle dX)^\frac{2^*_\lambda}{2}
    \leq C(\int_\Omega \delta^2 u_n^{2\delta-2} \langle A\nabla^*u_n,\nabla^*u_n\rangle dX)^\frac{2^*_\lambda}{2}
\end{align}
By using (\ref{w5}) and (\ref{g1}), we get
\begin{align*}
     \int_\Omega u_n^{\delta2^*_\lambda} dX&\leq C\{\frac{\delta^2}{(2\delta-1)}\|f\|_L^r(\Omega)\}^\frac{2^*_\lambda}{2}(\int_\Omega u_n^{(2\delta-\nu-1)r'}dX)^\frac{2^*_\lambda}{2r'}
\end{align*}
Choose $\delta$ such that $\delta2^*_\lambda=(2\delta-\nu-1)r'$ then $2^*_\lambda\delta=s$ . As $r<\frac{Q}{2}$ so $1-\frac{2^*_\lambda}{2r'}>0$. we have $\int_\Omega u_n^s dX\leq C$. Hence, by Dominated Convergence Theorem we get $u\in L^s(\Omega)$. 
\end{proof}

\subsection{The Case $\nu<1$}

%\begin{theorem}
%Let $\nu<1$ and $f\in L^r(\Omega),r=(\frac{2^*_\lambda}{1-\lambda})'$ is a nonnegative function (not identically zero ). Then (\ref{maineq}) has a unique solution.
%\end{theorem}
\textbf{Proof of Theorem \ref{Th5}:}
\begin{proof}
Since $\{u_n\}$ is bounded in $H^{1,\lambda}_0(\Omega)$ so it has a subsequence which converges to u weakly in $H^{1,\lambda}_0(\Omega)$. Without loss of generality we can assume $u_n\rightharpoonup u \text{in}\;H^{1,\lambda}_0(\Omega)$. Let $v\in C^1_0(\Omega)$.  By the Lemma \ref{lem1}, there exists $C>0$ such that $u_n(x)\geq C$ a.e $x\in\text{supp}(v)$ and for all $n\in \N$. So $$|\frac{f_nv}{(u_n+\frac{1}{n})^\nu}|\leq \frac{\|v\|_\infty|f|}{C^\nu}\;\mbox{for all}\;n\in\N$$ By the Dominated Convergence Theorem, we have 
   \begin{align}\label{DCT2}
       \lim_{n\to \infty}\int_\Omega \frac{f_nv}{(u_n+\frac{1}{n})^\nu}dX&= \int_\Omega \lim_{k\to \infty} \frac{f_nv}{(u_n+\frac{1}{n})^\nu}dX
       =\int_\Omega \frac{fv}{u^\nu}dX.
   \end{align}
   As $u_n$ is a solution of (\ref{equ}) so,
\begin{align*}
    \int_\Omega \langle A\nabla^*u_n,\nabla v\rangle dX=\int_\Omega \frac{f_nv}{(u_n+\frac{1}{n})^\nu} dX
\end{align*}
Take $n\to\infty$ and (\ref{DCT2}) we get
$$ \int_\Omega \langle A\nabla^*u,\nabla v\rangle dX=\int_\Omega \frac{fv}{u^\nu} dX$$
Hence, $u\in H^{1,\lambda}_0(\Omega)$ is a solution of (\ref{maineq}) with $\nu<1$. The proof of uniqueness is similar to Theorem \ref{Th1}.
\end{proof}
%\begin{theorem}
%If $u$ is a solution of (\ref{maineq}) with $\nu<1$ and $f\in L^r(\Omega), r\geq (\frac{2^*_\lambda}{1-\nu})'$ then
%\begin{enumerate}[label=(\roman*)]
 %   \item If $r>\frac{Q}{2}$ then $u\in L^\infty(\Omega)$\\
  %  \item if $(\frac{2^*_\lambda}{1-\nu})'\leq r<\frac{Q}{2}$ then $u\in L^s(\Omega)$
%\end{enumerate}
%where $s=\frac{Qr(\nu+1)}{(Q-2r)}, Q=(m+1)+\lambda m$ and $r'$ denotes the H\"{o}lder conjugate of $r$.
%\end{theorem}
\textbf{Proof of Theorem \ref{Th6}:}
\begin{proof}
(i) The proof is similar to the proof of Theorem \ref{Th2}.\\
(ii)
If $r=(\frac{2^*_\lambda}{1-\nu})'$ then $s=2^*_\lambda$. By the embedding theorem and (\ref{weakd}), we have 
\begin{align*}
    (\int_\Omega u_n^{2^*_\lambda}dX)^\frac{1}{2^*_\lambda}\leq C(\int_\Omega \langle A\nabla^*u_n,\nabla^*u_n\rangle dX)^\frac{1}{2}
    =C(\int_\Omega \frac{f_nu_n}{(u_n+\frac{1}{n})^\nu}dX)^\frac{1}{2}
    &\leq C(\int_\Omega fu_n^{1-\nu} dX)^\frac{1}{2}\\
    &\leq C \|f\|_{L^r(\Omega)}^\frac{1}{2}(\int_\Omega u_n^{(1-\nu)r'}dX)^\frac{1}{2r'}
    \end{align*}
    Since $r'=\frac{2^*_\lambda}{1-\nu}$ so using the above inequality we get 
    \begin{align*}
     \int_\Omega u_n^{2^*_\lambda} dX&\leq C\|f\|_{L^r(\Omega)}^\frac{2^*_\lambda}{1+\nu}
\end{align*}
 By Dominated Convergence Theorem we have $u\in L^{2^*_\lambda}(\Omega)$.\\
  Let $(\frac{2^*_\lambda}{1-\nu})'< r<\frac{Q}{2}$.
  Choose $\delta>1$ (to be determined later). We can treat the function $v=u_n^{2\delta-1}$ as a test function and put it in (\ref{weakd}), we obtain 
\begin{align}\label{M1}
    \int_\Omega\langle A\nabla^*u_n,\nabla^*u_n^{2\delta-1}\rangle dX=\int_\Omega \frac{f_nu_n^{2\delta-1}}{(u_n+\frac{1}{n})^\nu}dX\leq  \int_\Omega fu_n^{2\delta-\nu-1}dX\leq \|f\|_{L^r(\Omega)}(\int_\Omega u_n^{(2\delta-\nu-1)r'}dX)^\frac{1}{r'}
 \end{align}
 Also,
 \begin{align}\label{M2}
     \int_\Omega\langle A\nabla^*u_n,\nabla^*u_n^{2\delta-1}\rangle dX=\int_\Omega (2\delta-1)u_n^{2\delta-2}\langle A\nabla^*u_n,\nabla^*u_n\rangle dX=\int_\Omega \frac{(2\delta-1)}{\delta^2}\langle A\nabla^*u_n^\delta,\nabla^*u_n^\delta\rangle dX
 \end{align}
 Using (\ref{M1}) and (\ref{M2}) we have 
 $$\int_\Omega \langle A\nabla^*u_n^\delta,\nabla^*u_n^\delta\rangle dX)\leq \frac{\delta^2}{(2\delta-1)}\|f\|_{L^r(\Omega)}(\int_\Omega u_n^{(2\delta-\nu-1)r'}dX)^\frac{1}{r'} $$
By Theorem \ref{emb1}, there exists $C>0$ such that 
\begin{align*}
    \int_\Omega u_n^{\delta2^*_\lambda} dX&\leq C(\int_\Omega \langle A\nabla^*u_n^\delta,\nabla^*u_n^\delta\rangle dX)^\frac{2^*_\lambda}{2}\\
    &\leq C\{\frac{\delta^2}{(2\delta-1)}\|f\|_{L^r(\Omega)\}}^\frac{2^*_\lambda}{2}(\int_\Omega u_n^{(2\delta-\nu-1)r'}dX)^\frac{2^*_\lambda}{2r'}
\end{align*}
Choose $\delta$ such that $\delta2^*_\lambda=(2\delta-\nu-1)r'$ then $2^*_\lambda\delta=s$ . As $(\frac{2^*_\lambda}{1-\nu})'< r<\frac{Q}{2}$ so $\delta>1$ and $\frac{2^*_\lambda}{2r'}<1$.  Hence, we have $\int_\Omega u_n^s dX\leq C$. Hence, by Dominated Convergence Theorem, we get $u\in L^s(\Omega)$.
\end{proof}

%\begin{theorem}\label{W10}
%Let $\nu<1$ and $f\in L^r(\Omega)$ for some $r<\frac{2Q}{(Q+2)+\nu(Q-2)}$. Then there exists a solution $u$ of (\ref{maineq}), with $u$ in $W^{1,\lambda,q}_0(\Omega)$, $q=\frac{Qr(\nu+1)}{Q-r(1-\nu)}$.
%\end{theorem}
\textbf{Proof of Theorem \ref{Th7}:}
\begin{proof}
 Let $\epsilon<\frac{1}{n}$ and $v=(u_n+\epsilon)^{2\delta-1}-\epsilon^{2\delta-1}$ with $\frac{1+\nu}{2}\leq\delta<1$. We can treat $v$ as a function in $C^1_0(\Omega)$.
 Put $v$ in (\ref{weakd}) and we obtain
 \begin{align*}
     \int_\Omega\langle A\nabla^*u_n,\nabla^*u_n\rangle (u_n+\epsilon)^{2\delta-2} dX\leq \frac{1}{(2\delta-1)}\int_\Omega\frac{fv}{(u_n+\frac{1}{n})^\nu}
 \end{align*}
 As $\epsilon<\frac{1}{n}$ so we have
  \begin{align}\label{T3}
      \int_\Omega\langle A\nabla^*u_n,\nabla^*u_n\rangle (u_n+\epsilon)^{2\delta-2} dX\leq \frac{1}{(2\delta-1)}\int_\Omega f(u_n+\epsilon)^{2\delta-1-\nu}\; dX
  \end{align}
 By some simple calculation, we get
 \begin{align*}
     \int_\Omega\langle A\nabla^*v,\nabla^*v\rangle dX\leq \frac{\delta^2}{(2\delta-1)}\int_\Omega f(u_n+\epsilon)^{2\delta-1-\nu} dX 
 \end{align*}
 By Theorem \ref{emb1}, we have
 \begin{align*}
     (\int_\Omega v^{2^*_\lambda} dX)^\frac{2}{2^*_\lambda}\leq \frac{C\delta^2}{(2\delta-1)}\int_\Omega f(u_n+\epsilon)^{2\delta-1-\nu} 
 \end{align*}
 Take $\epsilon\rightarrow 0$ and use Dominated convergence Theorem we have,
 \begin{align}\label{T4}
     (\int_\Omega u_n^{2^*_\lambda\delta})^\frac{2}{2^*_\lambda}&\leq \frac{C\delta^2}{(2\delta-1)}\int_\Omega fu_n^{2\delta-1-\nu}
 \end{align}
 If $r=1$ then choose $\delta=\frac{\nu+1}{2}$ and from the previous inequality we have $\{u_n\}$ is bounded in $L^s(\Omega)$ with $s=\frac{Q(\nu+1)}{(Q-2)}$.\\
 If $r>1$ then choose $\delta$ in such a way that $(2\delta-1-\nu)r'=2^*_\lambda \delta$. Now, applying H\"{o}lder inequality on RHS of (\ref{T4}) we have,
 \begin{align*}
     (\int_\Omega u_n^{2^*_\lambda\delta})^\frac{2}{2^*_\lambda}&\leq \frac{C\delta^2}{(2\delta-1)}\|f\|_{L^r(\Omega)}(\int_\Omega u_n^{(2\delta-1-\nu)r'})^\frac{1}{r'}\\
     &=\frac{C\delta^2}{(2\delta-1)}\|f\|_{L^r(\Omega)}(\int_\Omega u_n^{2^*_\lambda \delta})^\frac{1}{r'}
 \end{align*}
 As $1\leq r<\frac{2Q}{(Q+2)+\nu(Q-2)}<\frac{Q}{2}$ so   $\frac{2}{2^*_\lambda}>\frac{1}{r'}$. Hence, $\{u_n\}$ is bounded in $L^s(\Omega)$ with $s=2^*_\lambda \delta=\frac{Qr(\nu+1)}{(Q-2r)}$. Using H\"{o}lder inequality in (\ref{T3}), we have
\begin{align*}
    \int_\Omega\langle A\nabla^*u_n,\nabla^*u_n\rangle (u_n+\epsilon)^{2\delta-2} dX\leq \frac{1}{(2\delta-1)}\|f\|_{L^r(\Omega)}(\int_\Omega (u_n+\epsilon)^{2^*_\lambda \delta})^\frac{1}{r'}
 \end{align*}
  Since $u_n$ is bounded in $L^s(\Omega)$ so
  $$\int_\Omega\langle A\nabla^*u_n,\nabla^*u_n\rangle (u_n+\epsilon)^{2\delta-2} dX\leq C. $$
 For $q=\frac{Qr(\nu+1)}{Q-r(1-\nu)}$ and above chosen $\delta$ satisfies the condition $(2-2\delta)q=(2-q)s$.\\
 So,
 \begin{align*}
     \int_\Omega \langle A\nabla^*u_n,\nabla^*u_n\rangle^\frac{q}{2} dX&=\int_\Omega \frac{|\sqrt{A}\nabla^*u_n|^q}{(u_n+\epsilon)^{q-q\delta}}(u_n+\epsilon)^{q-\delta q} dX\\
     &\leq(\int_\Omega\frac{|\sqrt{A}\nabla^*u_n|^2}{(u_n+\epsilon)^{2-2\delta}} dX) (\int_\Omega (u_n+\epsilon)^sdX)^{1-\frac{q}{2}}
 \end{align*}
 since $\{u_n\}$ is bounded in $L^s(\Omega)$ and $\epsilon<\frac{1}{n}$ so $\{u_n+\epsilon\}$ is bounded in  $L^s(\Omega)$. Consequently, $\{u_n\}$ is bounded in $W^{1,\lambda,q}_0(\Omega)$. Hence $u\in W^{1,\lambda,q}_0(\Omega)$.
\end{proof}
\begin{center}
    \section{VARIABLE SINGULAR EXPONENT}\label{sec7}
\end{center}

 Consider the equation 
 \begin{align}\label{Var1}
    -\Delta_\lambda u&=\frac{f}{u^{\nu(x)}} \text{ in }\Omega\nonumber\\
   &u>0 \;\text{in}\; \Omega\\
   &u=0 \;\text{on}\; \partial\Omega\nonumber
\end{align}
where $\nu\in C^1(\overline{\Omega})$ is a positive function.
\begin{theorem}
    Let $f\in L^{(2^*_\lambda)'}(\Omega)$ be a function. If there exists $K\Subset \Omega$ such that $0<\nu(x)\leq 1$ in $K^c$ (complement of K) then (\ref{Var1}) has an unique solution in $H_0^{1,\lambda}(\Omega)$ provided $\lambda\geq 1$.
\end{theorem}
\begin{proof}
The same approximation used in the earlier section yields the existence of a strictly positive function $u$, which is the increased limit of the sequence $\{u_n\}\subset H^{1,\lambda}_0(\Omega)\cap L^\infty(\Omega)$. Also, Lemma \ref{lem1} is satisfied. As $K\Subset \Omega$ so by Lemma \ref{lem1}, there exists $C>0$ such that $u_n(x)\geq C$ for a.e $x\in K\;\mbox{and for all}\; n\in\N$. For each $n\in\N$, $u_n$ solves 
\begin{align}\label{var2}
     -\Delta_\lambda u_n&=\frac{f_n}{(u_n+\frac{1}{n})^{\nu(x)}} \text{in}\Omega\nonumber\\
   &u>0 \;\text{in}\; \Omega\\
   &u=0 \;\text{on}\; \partial\Omega\nonumber
\end{align}
 By using  H\"{o}lder inequality and  the embedding theorem, we have 
\begin{align*}
   \int_\Omega \langle A\nabla^*u_n,\nabla^*u_n\rangle dx&=\int_\Omega \frac{f_nu_n}{(u_n+\frac{1}{n})^{\nu(x)}} dx\\
   &=\int_K \frac{f_nu_n}{(u_n+\frac{1}{n})^{\nu(x)}} dx +\int_{\{K^c\cap\Omega\}} \frac{f_nu_n}{(u_n+\frac{1}{n})^{\nu(x)}} dx\\
   &\leq ||\frac{1}{C^{\nu(x)}}||_\infty \int_K fu_n dx+ \int_{\{x\in K^c\cap\Omega: u_n(x)\leq 1\}} f u_n^{1-\nu(x)} dx + \int_{x\in K^c\cap\Omega: u_n(x)\geq 1} f u_n^{1-\nu(x)} dx\\
   &\leq ||\frac{1}{C^{\nu(x)}}||_\infty \int_K fu_n dx+ \int_{\{x\in K^c\cap\Omega: u_n(x)\leq 1\}} f dx + \int_{x\in K^c\cap\Omega: u_n(x)\geq 1} f u_n dx\\
   &\leq ||\frac{1}{C^{\nu(x)}}||_\infty ||f||_{L^{(2^*_\lambda)'}(\Omega)}||u_n||_{L^{2^*_\lambda}}+||f||_{L^1(\Omega)}+ ||f||_{L^{(2^*_\lambda)'}(\Omega)}||u_n||_{L^{2^*_\lambda}(\Omega)}\\
   &\leq C||f||_{L^{(2^*_\lambda)'}(\Omega)}||u_n||_{H^{1,\lambda}_0(\Omega)}+ ||f||_{L^1(\Omega)}
\end{align*}
 We obtained $$||u_n||_{H^{1,\lambda}_0(\Omega)}^2\leq  C||f||_{L^{(2^*_\lambda)'}(\Omega)}||u_n||_{H^{1,\lambda}_0(\Omega)}+ ||f||_{L^1(\Omega)}.$$
Hence, $u_n$ is bounded in $H^{1,\lambda}_0(\Omega)$. Without loss of generality we can assume that $u_n$ weakly converges to $u$ in $H^{1,\lambda}_0(\Omega)$. Let $w\in C^1_c(\Omega)$. Using  Lemma \ref{lem1}, there exists $c>0$ such that $u_n\geq c$ for a.e $x$ in $\text{supp}(w)$.
Since $u_n$ solves (\ref{var2}) so 
\begin{align*}
   \int_\Omega \langle A\nabla^*u_n,\nabla w\rangle dx&=\int_\Omega \frac{f_nw}{(u_n+\frac{1}{n})^{\nu(x)}} dx 
\end{align*}
Taking $n\to\infty$ and using the dominated convergence theorem, we get
$$  \int_\Omega \langle A\nabla^*u,\nabla w\rangle dx=\int_\Omega \frac{fw}{u^{\nu(x)}} dx$$
Hence, $u$ is a solution of (\ref{Var1}). The proof of the uniqueness part is identical to the one given in Theorem \ref{Th1}.
\end{proof}
\begin{theorem}
    Let $u$ be the solution of equation (\ref{Var1}) with $f\in L^r(\Omega)$, $r>\frac{Q}{2}$. Then $u\in L^\infty(\Omega)$, where $Q=(m+1)+\lambda m$.
\end{theorem}

\begin{proof}
   The proof is similar to that of the Theorem \ref{Th2} and is omitted here.
\end{proof}

\nocite{*}
\bibliographystyle{plain}
\bibliography{ref1.bib}
\end{document}